\documentclass[a4paper,12pt]{article}
\usepackage{a4wide}
\usepackage{amsmath}
\usepackage{amssymb}
\usepackage{amsthm}
\usepackage{latexsym}
\usepackage{graphicx}
\usepackage[english]{babel}
\usepackage{makeidx}
\usepackage{mathtools}
              
\newtheorem{dfn} [subsection]{Definition}
\newtheorem{obs} [subsection]{Remark}
\newtheorem{exm} [subsection]{Example}

\newtheorem{prop}[subsection]{Proposition}

\newtheorem{teor}[subsection]{Theorem}
\newtheorem{lema}[subsection]{Lemma}

\def\SL{\operatorname{SL}}

\def\Gal{\operatorname{Gal}}
\def\Lin{\operatorname{Lin}}

\def\Irr{\operatorname{Irr}}
\def\Reg{\operatorname{Reg}}

\def\ord{\operatorname{ord}}
\def\Supp{\operatorname{Supp}}
\def\Ker{\operatorname{Ker}}
\def\Ar{\operatorname{Ar}}
\def\Hol{\operatorname{Hol}}
\def\Sup{\operatorname{Sup}}
\def\lcm{\operatorname{lcm}}

\def\Cons{\operatorname{Cons}}
\numberwithin{equation}{section}

\begin{document}
\selectlanguage{english}
\frenchspacing

\large
\begin{center}
\textbf{On supercharacter theoretic generalizations of monomial groups and Artin's conjecture}

Mircea Cimpoea\c s and Alexandru F. Radu
\end{center}
\normalsize

\begin{abstract} 

We extend the notions of quasi-monomial groups and almost monomial groups, in the framework of
supercharacter theories, and we study their connection with Artin's 
conjecture regarding the holomorphy of Artin $L$-functions.


\end{abstract}

\section*{Introduction}

A group $G$ is called \emph{monomial}, if every complex irreducible character $\chi$ of $G$ is induced by a linear character $\lambda$ of
a subgroup $H$ of $G$, that is $\chi=\lambda^G$. A group $G$ is called \emph{quasi-monomial}, if for every irreducible character $\chi$ of $G$, 
there exists a subgroup $H$ of $G$ and a linear character $\lambda$ of $H$ such that $\lambda^G=d\chi$, where $d$ is a positive integer. 
A finite group $G$ is called {\em almost monomial} if for every distinct complex irreducible characters $\chi$ and $\psi$ of  $G$ 
there exist a subgroup $H$ of $G$ and a linear character $\lambda$ of $H$ such that the induced character $\lambda^G$ contains 
$\chi$ and does not contain  $\psi$. This definition, which generalizes quasi-monomial groups, appears \cite{monat} in connection with the 
study of the holomorphy of Artin L-functions associated to a finite Galois extension of $\mathbb Q$ at a point in the complex plane.
An equivalent characterization for almost monomial groups is given in Proposition \ref{113}.
 
Let $K/\mathbb Q$ be a finite Galois extension with Galois group $G$. For any character $\chi$ of $G$, let $L(s,\chi):=L(s,\chi,K/\mathbb Q)$
be the corresponding Artin L-function (\cite[P.296]{artin2}). Artin's conjecture states that $L(s,\chi)$ is holomorphic in $\mathbb C\setminus \{1\}$.
If the group $G$ is monomial (or quasi-monomial), then Artin's conjecture holds. 

Let $\Hol(s_0)$ be the semigroup of Artin $L$-functions, holomorphic at $s_0\in\mathbb C\setminus\{1\}$. F. Nicolae \cite{monat} proved that
if $G$ is almost monomial, then Artin's conjecture holds at $s_0$ if and only if $\Hol(s_0)$ is factorial. 
Let $\chi_1,\ldots,\chi_r$ be the complex irreducible characters of $G$, $f_1=L(s, \chi_1),\ldots,f_r=L(s,\chi_r)$ 
the corresponding Artin L-functions. In \cite{cim} it was proved that if $G$ is almost monomial and $s_0$ is not a common zero for 
any two distinct L-functions $f_k$ and $f_l$ then all Artin L-functions of $K/\mathbb Q$ are holomorphic at $s_0$. 
Also in \cite{cim}, some basic properties of almost monomial groups were stated.

The notion of a supercharacter theory for a finite group was introduced in 2008, by Diaconis and Isaacs \cite{isaacs}, as follows: 
A supercharacter theory of a finite group $G$, is a pair $C=(\mathcal X,\mathcal K)$ where $\mathcal X=\{X_1,\ldots,X_r\}$
is a partition of $\Irr(G)$, the set of irreducible characters of $G$, and $\mathcal K$ is a partition of $G$, such that: (1) $\{1\}\in\mathcal K$, 
(2) $|\mathcal X|=|\mathcal K|$ and (3) $\sigma_X:=\sum_{\psi\in X}\psi(1)\psi$ is constant for each $X\in\mathcal X$ and $K\in\mathcal K$. \pagebreak

The aim of our paper is to extend the notions of quasi-monomial and almost monomial groups in the framework of supercharacter theories, 
and to discuss the connections
with the supercharacter theoretic Artin conjecture, introduced by Wong \cite{wong}, which states that $L(s,\sigma_X)$ are 
holomorphic in $\mathbb C\setminus \{1\}$ for each $X\in \mathcal X$.

We say that $G$ is $C$-quasi-monomial, if for each $X\in\mathcal X$, there exists some subgroups $H_1,\ldots,H_t$ of $G$ and some linear characters
$\lambda_1,\ldots,\lambda_t$ such that $\lambda_1^G+\cdots+\lambda_t^G = d\sigma_X$, see Definition \ref{cqm-def}. We prove that the class of 
$C$-quasi-monomial groups is closed under factorization and taking direct products, see Theorem \ref{t25} and Theorem \ref{t28}. 
In Proposition \ref{doi}, we note that a $C$-quasi-monomial group satisfies the supercharacter theoretic Artin conjecture.

We say that $G$ is  \emph{$C$-almost monomial}, 
if for any $k\neq \ell$, there exist some subgroups $H_1,\ldots,H_t \leqslant G$ and 
linear characters $\lambda_1,\ldots,\lambda_t$ of $H_1,\ldots, H_t$ such that:
$\lambda_1^G+\cdots+\lambda_t^G=\sum_{i=1}^m \alpha_i\sigma_{X_i}$, where $\alpha_i\geq 0$ are integers with $\alpha_k>0$ and $\alpha_{\ell}=0$,
see Definition \ref{cam-def}.  We prove that the class of $C$-almost monomial
groups is closed under factorization and taking direct products, see Theorem \ref{t213} and Theorem \ref{t214}.

Let $F_1:=L(s,\sigma_{X_1}),\ldots, F_m:=L(s,\sigma_{X_m})$ and let $\Hol(C,s_0)$ be the semigroup of the functions
of the form $F:=F_1^{a_1}\cdots F_m^{a_m}$, where $a_i\geq 0$ are integers, which are holomorphic at $s_0$. 
In Theorem \ref{t11}, we prove that if $G$ is $C$-almost monomial, then the supercharacter theoretic Artin conjecture holds at $s_0$ 
if and only if $\Hol(C,s_0)$ is factorial. Also, in Theorem \ref{t22}, we prove that if $G$ is $C$-almost monomial and $s_0$ is not a common zero for any two distinct 
L-functions $F_{\ell}$ and $F_k$, where $k\neq \ell \in \{1,\ldots,m\}$, then the supercharacter theoretic Artin conjecture holds at $s_0$. These results,
generalize the aforementioned results on almost monomial groups.

\section{Preliminaries}

We recall that a finite group $G$ is monomial (or a $M$-group), if for any $\chi\in \Irr(G)$ then there exists a subgroup $H\leqslant G$
and a linear character $\lambda$ of $H$ such that $\lambda^G=\chi$. Any Abelian group $G$ is monomial, since all the irreducible
characters of $G$ are linear, but the converse is not true. 
According to Taketa's Theorem \cite{taketa}, every monomial group is solvable, but there are solvable groups
which are not monomial, the smallest example being $\SL(2,3)$. A slight generalization of monomial groups is the following:

\begin{dfn}\label{qm-def}
A finite group $G$ is called quasi-monomial (or a QM-group) if for any $\chi\in \Irr(G)$ then there exists a subgroup $H\leqslant G$
and a linear character $\lambda$ of $H$ such that $\lambda^G=d\chi$, where $d$ is a positive integer.
\end{dfn}

It is not known if there are quasi-monomial groups which are not monomial.

For a character $\psi$ of $G$, we denote $\Cons(\psi)$ the set of constituents of $\psi$.
We recall the following definition, which generalize quasi-monomial groups:

\begin{dfn}(\cite{monat})\label{am-def}
A finite group $G$ is called {\em almost monomial} (or AM-group) if for every two distinct characters $\chi \neq \psi \in \Irr(G)$, there exists a 
subgroup $H$ of $G$ and a linear character $\lambda$ of $H$ such $\chi\in \Cons(\lambda^G)$ and $\psi\notin \Cons(\lambda^G)$.
% $\langle \chi,\lambda^G \rangle >0$ and $\langle \psi ,\lambda^G \rangle=0$.
\end{dfn}

An important class of almost monomial groups are the symmetric groups, $S_n$, see \cite[Theorem 1.1]{cim}.
If $G$ is an almost monomial group and $N\unlhd G$ is a normal subgroup, then $G/N$ is almost monomial,
see \cite[Theorem 2.2]{cim}. Also, if $G,G'$ are finite groups, then $G\times G'$ is almost monomial if and only if
$G$ and $G'$ are almost monomial, see \cite[Theorem 2.3]{cim}.

The following result provides an equivalent form of Definition \ref{am-def} and shows that there is a 
kind of "duality" between the notions of quasi-monomial and almost monomial groups.

\begin{prop}\label{113}
Let $G$ be a finite group and assume that $\Irr(G)=\{\chi_1,\ldots,\chi_r\}$. Then, the following are equivalent:
\begin{enumerate}
\item[(1)] $G$ is almost monomial.
\item[(2)] For any $k\in\{1,\ldots,r\}$, there exists some subgroups $H_1,\ldots,H_m$ of $G$
and some linear characters $\lambda_1,\ldots,\lambda_m$ of $H_1,\ldots,H_m$ such that:
$$\Cons(\lambda_1^G+\cdots+\lambda_m^G)=\Irr(G)\setminus \{\chi_k\}.$$
\end{enumerate}
\end{prop}

\begin{proof}
$(1)\Rightarrow (2)$. Without loss of generality, we can assume that $k=r$.
According to Definition \ref{am-def}, 
for any $1\leq j\leq r-1=:m$, there exists a subgroup $H_{j}\leqslant G$ and a linear character $\lambda_{j}$ of $H_{j}$ such that
$\chi_j\in \Cons(\lambda_{j}^G)$ and $\chi_r \notin \Cons(\lambda_{j}^G)$.
It follows that 
$\Supp(\lambda_1^G+\cdots+\lambda_m^G)=\{\chi_1,\ldots,\chi_{r-1}\}$,
as required.

$(2)\Rightarrow (1)$. We fix $1\leq i\leq r$. Assume the condition $(2)$ is satisfied and let $k\neq i$. 
It follows that there exists $1\leq j\leq m$ such that $\chi_j\in \Cons(\lambda_j^G)$. On the other hand, 
$\chi_i\notin \Cons(\lambda_j^G)$, hence $G$ is almost monomial.
\end{proof}

\begin{exm}\label{ea5}\emph{
% Any quasi-monomial group $G$ is almost monomial in the sense of definition $1.2$. However, the converse is far
% from  being true.  
Let $A_5$ be the alternating group of order $5$. $A_5$ is a simple non-Abelian group, hence is not solvable.
Therefore $A_5$ is not monomial. However, $A_5$ is almost monomial:}
\emph{
We have that $\Irr(A_5)=\{\chi_1,\chi_2,\chi_3,\chi_4,\chi_5\}$, where $\chi_1$ is the trivial character, 
$\chi_2$ and $\chi_3$ are conjugated characters of degree $3$, $\chi_4$ has degree $4$ and $\chi_5$ has degree $5$. Obviously, $\chi_1$ is linear. Also one can check that 
$\chi_5$ is monomial. Let $H=\langle (12345) \rangle \subset A_5$, which is isomorphic to the cyclic group of order $5$. The characters induced from the irreducible 
(linear) characters of $H$ are $\chi_1+\chi_2+\chi_3+\chi_5$, $\chi_2+\chi_4+\chi_5$ and $\chi_3+\chi_4+\chi_5$.}

\emph{
Let $U=\langle (12)(45),(345) \rangle \subset A_5$, which is isomorphic to $S_3$. Let $\psi:U\rightarrow \{\pm 1\}$ be the sign function on $U$, 
which is a linear character. We have that $\psi^{A_5} = \chi_2+\chi_3+\chi_4$. From Proposition \ref{113} it follows that $A_5$ is almost monomial.}
\end{exm}

Let $K/\mathbb Q$ be a finite Galois extension. 
For the character $\chi$ of a representation of the Galois group $G:=\Gal(K/\mathbb Q)$
on a finite dimensional complex vector space let $L(s,\chi):=L(s,\chi,K/\mathbb Q)$ be the corresponding Artin L-function 
(\cite[P.296]{artin2}). 
Artin conjectured that $L(s,\chi)$ is holomorphic in $\mathbb C\setminus \{1\}$.   
Brauer proved that $L(s,\chi)$ is meromorphic in $\mathbb C$.
Let $\chi_1,\ldots,\chi_r$ be the irreducible characters of $G$, $f_1=L(s, \chi_1),\ldots,f_r=L(s,\chi_r)$ the corresponding Artin L-functions.

For two characters $\phi$ and $\psi$ of $G$,
$L(s,\phi+\psi) = L(s,\phi) \cdot L(s,\psi)$, so the set of L-functions corresponding to all characters
of $G$ is a multiplicative semigroup, denoted by $\Ar$. 

Since any character of $G$ is a linear 
combination with positive integer coefficients 
of irreducible characters, the semigroup $\Ar$ is generated by $f_1,\ldots,f_r$, that is
$$\Ar:=\{f_1^{k_1}\cdot\ldots\cdot f_r^{k_r}\mid k_1\geq 0,\ldots,k_r\geq 0\}.$$
Since $f_1,\ldots,f_r$ are multiplicatively independent, see \cite[Satz 5, p. 106]{artin1}, it follows
that $\Ar$ is factorial of rank $r$; in other words, $\Ar$ is isomorphic to $\mathbb Z_{\geq 0}^r$.
Moreoever, F.\ Nicolae \cite{crelle} proved that $f_1,\ldots,f_r$ are algebraically independent over $\mathbb C$,
a result extended later in \cite{forum}, where it was proved that $f_1,\ldots,f_r$ are algebraically independent 
over the field of meromorphic functions of order $<1$.

For $s_0\in\mathbb C,s_0\neq 1$ let $\Hol(s_0)$ be the
subsemigroup of $Ar$ consisting of the L-functions which are holomorphic at
$s_0$. F.\ Nicolae \cite{numb} proved that $\Hol(s_0)$ is an 
affine subsemigroup of $\Ar$, isomorphic to an affine subsemigroup of $\mathbb Z_{\geq 0}^r$.
Artin's conjecture at $s_0$ can be stated as:
$\Hol(s_0)=Ar$. We end this section by recalling the following results:

%In (\cite{monat}) it was proved the following:

\begin{teor}(\cite{monat})\label{t1}
If $G=\Gal(K/\mathbb Q)$ is almost monomial, then the following assertions are equivalent:
\begin{enumerate}
\item[1)] Artin's conjecture is true: $\Hol(s_0)=\Ar.$
\item[2)] The semigroup $\Hol(s_0)$ is factorial.
\end{enumerate}
\end{teor}

%In (\cite{cim}) it was proved the following:

\begin{teor}(\cite{cim})\label{t2}
If $G$ is almost monomial and $s_0$ is not a common zero for any two distinct L-functions $f_k$ and $f_l$ then all Artin L-functions of $K/\mathbb Q$ 
are holomorphic at $s_0$.  
\end{teor}

\section{Supercharacter theoretic quasi and almost monomial groups}

Diaconis and Isaacs \cite{isaacs} introduced the theory of supercharacters as follows:

\begin{dfn}
Let $G$ be a finite group. Let $\mathcal K$ be a partition of $G$ and let $\mathcal X$ be a partition of $\Irr(G)$. 
The ordered pair $C:=(\mathcal X,\mathcal K)$
is a \emph{supercharacter theory} if:
\begin{enumerate}
\item $\{1\}\in\mathcal K$,
\item $|\mathcal X|=|\mathcal K|$, and
\item for each $X\in\mathcal X$, the character $\sigma_X=\sum_{\psi\in \mathcal X}\psi(1)\psi$ is constant on each $K\in\mathcal K$.
\end{enumerate}
The characters $\sigma_X$ are called \emph{supercharacters}, and the elements $K$ in $\mathcal K$ are called \emph{superclasses}.
We denote $\Sup(G)$ the set of supercharacter theories of $G$.
\end{dfn}

Diaconis and Isaacs showed their theory enjoys properties similar to the classical character theory.
For example, every superclass is a union of conjugacy classes in $G$; see \cite[Theorem 2.2]{isaacs}.
The irreducible characters and conjugacy classes of $G$ give a supercharacter 
theory of $G$, which will be referred to as the \emph{classical theory} of $G$.

Also, as noted in \cite{isaacs}, every group $G$ admits
a non-classical theory with only two supercharacters $1_G$ and $\Reg(G)-1_G$ and
superclasses $\{1\}$ and $G\setminus\{1\}$, where $1_G$ denotes the trivial character of $G$
and $$\Reg(G) = \sum_{\chi\in\Irr(G)}\chi(1)\chi$$ is the regular character of $G$. 
This theory will be called the \emph{maximal theory} of $G$.

Let $C=(\mathcal X,\mathcal K)$ and $C'=(\mathcal X',\mathcal K')$ be two supercharacter theories of $G$.
We write $\mathcal X\preceq \mathcal X'$ if every $X\in \mathcal X$ is a subset of some $X'\in\mathcal X'$.
This is equivalent to saying that any $X'\in \mathcal X'$ is a union of parts of $\mathcal X$. According to
\cite[Corollary 3.4]{hend}, $\mathcal X\preceq \mathcal X'$ if and only if $\mathcal K\preceq \mathcal K'$.

\begin{dfn}(\cite[Definition 3.4]{hend})
We say that $C\preceq C'$ if $\mathcal X\preceq \mathcal X'$.
\end{dfn}

The set $(\Sup(G),\preceq)$ forms a poset with the minimal element being the classical theory of $G$
and the maximal element being the maximal theory of $G$.

We introduce the following generalization of Definition \ref{qm-def}:

\begin{dfn}\label{cqm-def}
Let $G$ be a finite group and let $C:=(\mathcal X,\mathcal K)$ be a supercharacter theory on $G$. Assume that $\mathcal X=\{X_1,\ldots,X_m\}$.
We say that $G$ is \emph{$C$-quasi-monomial} (or a $C-QM$-group), 
if for any $k\in \{1,\ldots,m\}$, there exists some subgroups $H_1,\ldots,H_t \leqslant G$ (not necessarily distinct) and 
linear characters $\lambda_1,\ldots,\lambda_t$ of $H_1,\ldots, H_t$ such that:
$$\lambda_1^G+\cdots+\lambda_t^G=d\sigma_{X_k},$$ 
where $d$ is a positive integer.
\end{dfn}

\begin{prop}\label{unu}
Let $G$ be a finite group. Then the following hold:
\begin{enumerate}
\item[(1)] If $(\mathcal X,\mathcal K)$ is the classical theory of $G$, then $G$ is quasi-monomial in the sense of Definition \ref{qm-def} if and only if
      $G$ is $C$-quasi-monomial in the sense of Definition \ref{cqm-def}.
\item[(2)] If $C,C'\in\Sup(G)$ with $C\preceq C'$ and $G$ is $C$-quasi-monomial then $G$ is $C'$-quasi-monomial.
\item[(3)] If $C$ is the maximal theory of $G$, then $G$ is $C$-quasi-monomial.			
\end{enumerate}
\end{prop}

\begin{proof}
$(1)$ and $(2)$ are obvious.

$(3)$ According to the Aramata-Brauer Theorem, see \cite{arama} and \cite{brauer}, $\Reg(G)-1_G$ can be written as a positive rational linear
combination of induced linear characters. It follows that there exists some subgroups $H_1,\ldots,H_t \leqslant G$ (not necessarily distinct) and 
linear characters $\lambda_1,\ldots,\lambda_t$ of $H_1,\ldots, H_t$ such that $$\lambda_1^G+\cdots+\lambda_t^G=d(\Reg(G)-1_G),$$
where $d$ is a positive integer. On the other hand, $(1_G)^G=1_G$. Thus, we get the required result.
\end{proof}

Let $G$ be a finite group and let $N\unlhd G$ be a normal subgroup of $G$.
It is well known that $\Irr(G/N)$ is in bijection with the set 
$$\{\chi\in \Irr(G)\;:\;N\subset \Ker(\chi)\}.$$
For a character $\widetilde \chi \in \Irr(G/N)$,
we denote $\chi$ the corresponding character in $\Irr(G)$, that is $\chi(g):=\widetilde{\chi}(\hat g)$ 
for all $g\in G$, where $\hat g$ is the class of $g$ in $G/N$.

\begin{lema}\label{leman}
Let $G$ be a finite group, $H\leqslant G$ a subgroup and $N\unlhd G$ a normal subgroup. Let $\lambda$ be a linear character of $H$ such that
$N\subset \Ker(\lambda^G)$. Then:
\begin{enumerate}
\item[(1)] $H\cap N \subset \Ker(\lambda)$, hence $\widetilde{\lambda}:\frac{HN}{N} \to \mathbb C^*$, $\widetilde{\lambda}(hN):=\lambda(h)N$, is a linear
      character of $\frac{HN}{N}$.
\item[(2)] For any character $\chi$ of $G$ with $N\subset \Ker(\chi)$, we have that 
      $\langle \widetilde{\lambda}^{G/N},\widetilde{\chi} \rangle = \langle \lambda^G,\chi \rangle$.
\end{enumerate}
\end{lema}

\begin{proof}
$(1)$ We assume that $H\cap N \nsubseteq \Ker(\lambda)$. Then $\lambda_{H\cap N}$ is a nontrivial linear character of $H\cap N$.
   On the other hand, since $N\subset \Ker(\lambda^G)$, it follows that $(\lambda^G)_{H\cap N} = |G:H|1_{H\cap N}$. Therefore,
	 by Frobenius reciprocity, we have that:
	 \begin{equation}\label{ee1}
	 \langle (\lambda_{H\cap N})^H, (\lambda^G)_{H}  \rangle = \langle \lambda_{H\cap N}, (\lambda^G)_{H\cap N} \rangle = 0.
   \end{equation}
	 On the other hand, we have that:
	 \begin{equation}\label{ee2}
	 \langle (\lambda_{H\cap N})^H, \lambda  \rangle = \langle \lambda_{H\cap N}, \lambda_{H\cap N} \rangle = 1.
   \end{equation}
	 From \eqref{ee1} and \eqref{ee2} it follows that
	 $$ \langle \lambda, (\lambda^G)_H \rangle = \langle \lambda^G,\lambda^G \rangle = 0,$$
	 and we get a contradiction.
	
$(2)$ By Frobenius reciprocity, we have that:
\begin{align*}
& \langle \widetilde{\lambda}^{G/N},\widetilde{\chi} \rangle = \langle \widetilde{\lambda},\widetilde{\chi}|_{HN/N} \rangle = \frac{|H\cap N|}{|H|} 
\sum_{\widetilde h\in HN/N}\widetilde{\lambda}(\widetilde h)\overline{\widetilde{\chi}(\widetilde h)} = \\
& = \frac{1}{H} \sum_{h\in H} \lambda(h)\overline{\chi(h)} = \langle \lambda, \chi_H \rangle = \langle \lambda^G, \chi \rangle,
\end{align*}
hence we are done.
\end{proof}

Let $G$ be a finite group and let $C:=(\mathcal X,\mathcal K)$ be a supercharacter theory of $G$. Let $N\unlhd G$ be a normal subgroup of $G$.
The group $N$ is called \emph{$C$-normal} or \emph{supernormal}, if $N$ is a union of superclasses from $C$; see \cite{marberg} and \cite{hend}.
Let $X\in \mathcal X$. By abuse of notation, we write $X\subset \Irr(G/N)$ if for any $\chi \in X$, then $N\subset \Ker(\chi)$.
Let $K\in \mathcal K$. We denote $\widetilde K:=KN/N \subset G/N$.

Now, assume that $N$ is $C$-normal. Without loss of generality, we can assume that $X_i\subset \Irr(G/N)$ for $1\leq i\leq p$ and $X_i \subsetneq \Irr(G/N)$ for 
$p+1\leq i\leq m$. Let $\widetilde{\mathcal X}:=\{\widetilde{X_1},\ldots,\widetilde{X_m}\}$ and $\widetilde{\mathcal K}:=\{\widetilde K_1,\ldots,\widetilde K_m\}$.
According to \cite[Proposition 6.4]{hend}, the pair:
$$\widetilde C:=C^{G/N}=(\widetilde{\mathcal X},\widetilde{\mathcal K})$$
is a supercharacter theory of $G/N$.

\begin{teor}\label{t25}
With the above notations, if $G$ is $C$-quasi-monomial and $N\unlhd G$ is a $C$-normal subgroup of $G$, then $G/N$ is $C^{G/N}$-quasi-monomial.
\end{teor}

\begin{proof}
Let $\widetilde{X_k} \in \widetilde{\mathcal X}$. Since $G$ is $C$-quasi-monomial, there exist some subgroups $H_1,\ldots,H_t \leqslant G$ and some
linear characters $\lambda_1,\ldots,\lambda_t$ of $H_1,\ldots, H_t$ such that
$$\lambda_1^G+\cdots+\lambda_t^G=d\sigma_{X_k},$$ 
where $d$ is a positive integer. We fix an index $i$ with $1\leq i\leq t$.
Since $\Cons(\lambda_i^G)\subset X_k$ and for any $\chi\in X_k$, we have $N\subset \Ker(\chi)$, it follows that
$N\subset \Ker(\lambda_i^G)$. From Lemma \ref{leman}(1), it follows that $H_i\cap N \subset \Ker(\lambda_i)$ and thus
$\widetilde{\lambda_i}:H_iN/N\to\mathbb C^*$, $\widetilde{\lambda_i}(h_iN)=\lambda_i(h_i)$, is a linear character of the subgroup $H_iN/N$ of $G/N$.
 From Lemma \ref{leman}(2) and straightforward computations, it follows that:
$$\widetilde{\lambda_1}^{G/N}+\cdots+\widetilde{\lambda_t}^{G/N}=d\sigma_{\widetilde{X_k}},$$
and thus $G/N$ is $C$-quasi-monomial.
\end{proof}

We recall the following result:

\begin{lema}(\cite[Proposition 8.1]{hend})
Let $G$ and $G'$ be two finite groups and let $C=(\mathcal X,\mathcal K)$ and $C'=(\mathcal X',\mathcal K')$ be supercharacter theories of $G$ and $G'$, respectively.
Then $C\times C'=(\mathcal X\times\mathcal X',\mathcal K\times\mathcal K')$ is a supercharacter theory of the direct product $G\times G'$.
\end{lema}

\begin{teor}\label{t28}
Let $G$ and $G'$ be two finite groups and let $C=(\mathcal X,\mathcal K)$ and $C'=(\mathcal X',\mathcal K')$ be supercharacter theories of $G$ and $G'$, respectively.
Then the following are equivalent:
\begin{enumerate}
\item[(1)] $G$ is $C$-quasi-monomial and $G'$ is $C'$-quasi-monomial.
\item[(2)] $G\times G'$ is $C\times C'$-quasi-monomial.
\end{enumerate}
\end{teor}

\begin{proof}
$(1)\Rightarrow (2)$ Let $X\in\mathcal X$ and $X'\in\mathcal X'$. From hypothesis, there exist $H_1,\ldots,H_t\leqslant G$, $\lambda_1,\ldots,\lambda_t$ linear
characters of $H_1,\ldots,H_t$ such that $$\lambda_1^G+\cdots+\lambda_t^G=d\sigma_{X_k},$$
where $d\geq 1$ is an integer. Also, there exist $H_1',\ldots,H_t'\leqslant G'$, $\mu_1,\ldots,\mu_{t'}$ linear
characters of $H_1',\ldots,H_{t'}'$ such that $$\mu_1^{G'}+\cdots+\mu_{t'}^{G'}=d'\sigma_{X'},$$
where $d'\geq 1$ is an integer. We consider the subgroups $H_i\times H'_{i'}$ of $G\times G'$ and
the linear characters $\lambda_i\times\mu_{i'}$ of $H_i\times H'_{i'}$, where $1\leq i\leq t$ and $1\leq i'\leq t'$.
A straightforward computations shows that $$\sum_{i=1}^{t}\sum_{i'=1}^{t'} (\lambda_i\times\mu_{i'})^{G\times G'} = dd'\sigma_{X\times X'},$$
thus $G\times G'$ is $C\times C'$-quasi-monomial.

$(2) \Rightarrow (1)$ The group $G'$ can be seen as a $C\times C'$-normal subgroup of $G\times G'$, hence, by Theorem \ref{t25}, 
it follows that $G$ is $C$-quasi-monomial.
\end{proof}

We introduce the following generalization of both Definition \ref{am-def} and Definition \ref{cqm-def}:

\begin{dfn}\label{cam-def}
Let $G$ be a finite group and let $C=(\mathcal X,\mathcal K)$ be a supercharacter theory on $G$. Assume that $\mathcal X=\{X_1,\ldots,X_m\}$.
We say that $G$ is 
\emph{$C$-almost monomial}, 
if for any $k\neq \ell$, there exist some subgroups $H_1,\ldots,H_t \leqslant G$ (not necessarily distinct) and 
linear characters $\lambda_1,\ldots,\lambda_t$ of $H_1,\ldots, H_t$ such that:
$$\lambda_1^G+\cdots+\lambda_t^G=\sum_{i=1}^m \alpha_i\sigma_{X_i},$$ 
where $\alpha_i\geq 0$ are integers with $\alpha_k>0$ and $\alpha_{\ell}=0$.
\end{dfn}

\begin{prop}
If $G$ is finite group and $C=(\mathcal X,\mathcal K)$ is the classical theory on $G$, then $G$ is $C$-almost monomial,
in the sense of Definition \ref{cam-def}, if and only if $G$ is almost monomial in the sense of Definition \ref{am-def}.
\end{prop}

\begin{proof}
Assume that $\Irr(G)=\{\chi_1,\ldots,\chi_r\}$, $d_j=\chi_j(1)$ for $j\in\{1,\ldots,r\}$, and let
$d=\lcm(d_1,\ldots,d_r)$. If $G$ is almost monomial, then for any $k\neq \ell$, there exists a subgroup 
$H$ of $G$ and a linear character $\lambda$ of $H$ such $\chi_k\in \Cons(\lambda^G)$ and 
$\chi_{\ell}\notin \Cons(\lambda^G)$. Then 
$$d\lambda^G = \alpha_1 (d_1\chi_1) + \cdots +\alpha_r (d_r\chi_r),$$ 
for some integers $\alpha_j\geq 0$ with $\alpha_k>0$ and $\alpha_{\ell}=0$.

Conversely, if $G$ is $(\mathcal X,\mathcal K)$-almost monomial, then there exist $H_1,\ldots,H_t \leqslant G$, subgroups of $G$,
and linear characters $\lambda_1,\ldots,\lambda_t$ of $H_1,\ldots, H_t$ such that:
$$\lambda_1^G+\cdots+\lambda_t^G=\alpha_1 (d_1\chi_1) + \cdots +\alpha_r (d_r\chi_r),$$ 
where $\alpha_j\geq 0$ are integers with $\alpha_k>0$ and $\alpha_{\ell}=0$. In particular, $\chi_{\ell}\notin \Cons(\lambda_j^G)$
for any $j\in\{1,\ldots,r\}$ and there exists $j_0\in\{1,\ldots,r\}$ with $\chi_k\in\Cons(\lambda_{j_0}^G)$. We choose $H=H_{j_0}$
and $\lambda=\lambda_{j_0}$ and we note that $\chi_k\in\Cons(\lambda^G)$ and $\chi_{\ell}\notin \Cons(\lambda^G)$. Thus, $G$ is
almost monomial.
\end{proof}

The following result, generalizes Proposition \ref{113} and its proof is similar to the proof of Proposition \ref{113}, so we skip it.

\begin{prop}\label{313}
Let $G$ be a finite group and let $C=(\mathcal X,\mathcal K)$ be a supercharacter theory on $G$, where $\mathcal X=\{X_1,\ldots,X_m\}$.
Then, the following are equivalent:
\begin{enumerate}
\item[(1)] $G$ is $C$-almost monomial.
\item[(2)] For any $k\in\{1,\ldots,m\}$, there exists some subgroups $H_1,\ldots,H_s$ of $G$
and some linear characters $\lambda_1,\ldots,\lambda_s$ of $H_1,\ldots,H_s$ such that:
$$\lambda_1^G+\cdots+\lambda_m^G=\alpha_1\sigma_{X_1}+\cdots+\alpha_{k-1}\sigma_{X_{k-1}}+\alpha_{k+1}\sigma_{X_{k-1}}+\cdots +\alpha_m\sigma_{X_{m}},$$
where $\alpha_i> 0$ are some integers.
\end{enumerate}
\end{prop}

%\begin{proof}
%The proof is similar to the proof of Proposition \ref{113}, so we skip it.
%\end{proof}

% \begin{obs}
% \emph{Let $(\mathcal X=\{\{1\},\;G\setminus\{1\}\},\mathcal K=\{1_G,\;\Reg(G)-1_G\})$ be the maximal theory of the finite group $G$.
% If we take $H=G$ and $\lambda=1_G$, then $1_G^G=1_G=1\cdot 1_G+0\cdot \Reg(G)$, 
% $\langle 1_G,1_G \rangle = 1$ and $\langle 1_G,\Reg(G)-1_G \rangle = 0$. However, it is not necessarily true that we can find some
% subgroups $H_1,\ldots,H_t$ of $G$ and some linear characters $\lambda_1,\ldots,\lambda_t$ such that $\lambda_1^G+\cdots+\lambda_t^G$
% is a multiple of $\Reg(G)$. Hence, we cannot be sure that $G$ is $(\mathcal X,\mathcal K)$-almost monomial.}
% \end{obs}

For a finite group $G$, we may ask if $C,C'\in\Sup(G)$ with $C\preceq C'$ and $G$ is $C$-almost monomial then $G$ is $C'$-almost monomial also,
as in the quasi-monomial case(see Proposition \ref{unu}(2)) The following example shows that this phenomenon is not always true:

\begin{exm}
\emph{
Let $G=SL(2,3)$ be the special linear group of degree two over a field of three elements. It is well known that $G$ is solvable, but
it is not monomial. However, $G$ is almost monomial. $G$ has $7$ irreducible characters: $\chi_1=1_G,\chi_2,\chi_3$ are linear, $\chi_4,\chi_5,\chi_6$ have
the degree $2$ and $\chi_7$ has the degree $3$. The characters $\chi_1,\chi_2,\chi_3$ and $\chi_7$ are monomial, but $\chi_4,\chi_5$ and $\chi_6$ are not.
Also, $\chi_5=\chi_2\chi_4$ and $\chi_6=\chi_3\chi_4$. Moreover, $\chi_{45}:=\chi_4+\chi_5$, $\chi_{46}:=\chi_4+\chi_6$ and $\chi_{456}:=\chi_4+\chi_5+\chi_6$ are monomial,
and any monomial character of $G$ is a linear combination of $\chi_1,\chi_2,\chi_3,\chi_7,\chi_{45},\chi_{46}$ and $\chi_{456}$.
We let: 
$$\mathcal X:=\{X_1:=\{\chi_1\},X_2:=\{\chi_2,\chi_3\},X_3:=\{\chi_4\},X_4:=\{\chi_5,\chi_6\},X_5:=\{\chi_7\}\}.$$
One can easily check that there exists a partition $\mathcal K$ of $G$ such that $C=(\mathcal X,\mathcal K)$ is a supercharacter theory of $G$ ($\mathcal K$ is
the set of classes for the equivalence relation $g \sim g'$ if and only if $\sigma_{X_i}(g)=\sigma_{X_j}(g')$ for all $1\leq i,j\leq 5$).}

\emph{We claim that $G$ is not
$C$-almost monomial. Indeed, we cannot find subgroups $H_1,\ldots,H_t$ and linear characters $\lambda_1,\ldots,\lambda_t$ of $H_1,\ldots,H_t$ such that
$$\lambda_1^G+\cdots+\lambda_t^G = \alpha_1 \sigma_{X_1} + \alpha_2 \sigma_{X_2} + \alpha_3 \sigma_{X_3} + \alpha_5 \sigma_{X_5},$$
with $\alpha_3>0$, since for any $H \leqslant G$ and $\lambda\in \Lin(H)$ with $\chi_4\in\Cons(\lambda^G)$, one has $\chi_5\in \Cons(\lambda^G)$ or $\chi_6\in \Cons(\lambda^G)$.
Contradiction by Proposition \ref{313}.}

\emph{We let: 
$\mathcal X':=\{X'_1:=\{\chi_1\},X'_2:=\{\chi_2,\chi_3\},X'_3:=\{\chi_4,\chi_5,\chi_6\},X'_4:=\{\chi_7\}\}$. Then, there exists
a partition $\mathcal K'$ of $G$ such that $C'=(\mathcal X',\mathcal K')$ is a supercharacter theory of $G$. Since $\chi_1$, $\chi_2$,
$\chi_3$, $\chi_{456}$ and $\chi_7$ are monomial, it follows that $G$ is $C'$-quasi-monomial.}
\end{exm}

The following result generalizes \cite[Theorem 2.2]{cim} and Theorem \ref{t25}:

\begin{teor}\label{t213}
Let $G$ be a finite group and let $C=(\mathcal X,\mathcal K)$ be a supercharacter theory of $G$. 
Let $N\unlhd G$ be a $C$-normal subgroup of $G$. If $G$ is $C$-almost monomial, then $G/N$ is $C^{G/N}$-almost monomial.
\end{teor}

\begin{proof}
Let $\widetilde{X_k} \in \widetilde{\mathcal X}$. Since $G$ is $C$-almost monomial, from Proposition \ref{313}, it follows that
there exist some subgroups $H_1,\ldots,H_s \leqslant G$ and some
linear characters $\lambda_1,\ldots,\lambda_s$ of $H_1,\ldots, H_s$ such that
$$\lambda_1^G+\cdots+\lambda_t^G=\alpha_1\sigma_{X_1}+\cdots+\alpha_{k-1}\sigma_{X_{k-1}}+\alpha_{k+1}\sigma_{X_{k+1}}+\cdots +\alpha_m\sigma_{X_{m}}.$$
As in the proof of Theorem \ref{t25}, we can define the linear characters $\widetilde{\lambda_j}$ of $H_jN/N$, and we have:
$$\widetilde{\lambda_1}^G+\cdots+\widetilde{\lambda_t}^G=\alpha_1\sigma_{\widetilde{X_1}}+\cdots+\alpha_{k-1}\sigma_{\widetilde{X_{k-1}}}+
\alpha_{k+1}\sigma_{\widetilde{X_{k-1}}}+\cdots +\alpha_m\sigma_{\widetilde{X_{m}}},$$
and thus $G/N$ is $C^{G/N}$-almost monomial.
\end{proof}

The following result generalizes \cite[Theorem 2.3]{cim} and Theorem \ref{t28}:

\begin{teor}\label{t214}
Let $G$ and $G'$ be two finite groups and let $C=(\mathcal X,\mathcal K)$ and $C'=(\mathcal X',\mathcal K')$ be supercharacter theories of $G$ and $G'$, respectively.
Then the following are equivalent:
\begin{enumerate}
\item[(1)] $G$ is $C$-almost monomial and $G'$ is $C'$-almost monomial.
\item[(2)] $G\times G'$ is $C\times C'$-almost monomial.
\end{enumerate}
\end{teor}

\begin{proof}
$(1)\Rightarrow(2)$. Assume that $\mathcal X=\{X_1,\ldots,X_m\}$ and $\mathcal X'=\{X'_1,\ldots,X'_{m'}\}$. 
We fix $$(k,k') \in \{1,\ldots,m\}\times\{1,\ldots,m'\}.$$
Since $G$ is $C$-almost monomial, from Proposition \ref{313} it follows that there exists
some subgroups $H_1,\ldots,H_s$ of $G$, some linear characters $\lambda_1,\ldots,\lambda_s$ of $H_1,\ldots,H_s$,
and some positive integers $\alpha_i$ such that:
\begin{equation}\label{ecoo1}
\lambda_1^G+\cdots+\lambda_s^G=\alpha_1\sigma_{X_1}+\cdots+\alpha_{k-1}\sigma_{X_{k-1}}+\alpha_{k+1}\sigma_{X_{k+1}}+\cdots +\alpha_m\sigma_{X_{m}}.
\end{equation}
Similarly, there exists
some subgroups $H'_1,\ldots,H'_{s'}$ of $G'$, some linear characters $\lambda_1,\ldots,\lambda_{s'}$ of $H'_1,\ldots,H'_{s'}$,
and some positive integers $\alpha'_i$ such that:
\begin{equation}\label{ecoo2}
\lambda_1'^{G'}+\cdots+\lambda_{s'}'^{G'}=\alpha'_1\sigma_{X'_1}+\cdots+\alpha'_{k'-1}\sigma_{X'_{k'-1}}+\alpha'_{k'+1}\sigma_{X'_{k'+1}}+\cdots +\alpha_m\sigma_{X'_{m}}.
\end{equation}
If $\mathbf 1$ is the unique character of the trivial subgroup of $G$, and $\mathbf 1'$ is the unique character of the trivial subgroup of $G'$, then
\begin{equation}\label{unnu}
\mathbf 1^G=\Reg(G)=\sigma_{X_1}+\cdots+\sigma_{X_m},\;\mathbf 1'^{G'}=\Reg(G')=\sigma_{X'_1}+\cdots+\sigma_{X'_{m'}}.
\end{equation}
By straightforward computations, from \eqref{ecoo1}, \eqref{ecoo2} and \eqref{unnu}, it follows that:
\begin{align*}
& (\lambda_1\times 1_{G'})^{G\times G'}+\cdots+(\lambda_s\times 1_{G'})^{G\times G'}+
  (1_G\times \lambda'_1)^{G\times G'}+\cdots+ (1_G\times \lambda'_{s'})^{G\times G'}= \\
& = \sum_{i=1}^m \sum_{i'=1,\;i'\neq k'}^{m'}	\alpha_{i'}\sigma_{X_i\times X'_{i'}} +
    \sum_{i=1,\;i\neq k}^m \sum_{i'=1}^{m'}	\alpha_{i}\sigma_{X_i\times X'_{i'}} = \sum_{i=1}^m\sum_{i'=1}^{m'} a_{ii'}\sigma_{X_i\times X'_{i'}}.
\end{align*}
Note that $a_{ii'}>0$ for all $(i,i')\in\{1,\ldots,m\}\times\{1,\ldots,m'\}$ with $(i,i')\neq (k,k')$ and $a_{kk'}=0$. Therefore, from Proposition \ref{313},
it follows that $G\times G'$ is $C\times C'$-almost monomial.

$(2)\Rightarrow(1)$. Follows from Theorem \ref{t213}, using a similar argument as in the proof of Theorem \ref{t28}.
\end{proof}

\newpage
\section{Supercharacter theoretic Artin conjecture}

Let $G$ be a finite group.
Let $C=(\mathcal X,\mathcal K)\in \Sup(G)$ be a supercharacter theory of $G$.% and let $\Sup(G)$ denote the set of supercharacter theories of $G$.% with respect to $(\mathcal X,\mathcal K)$. 
We consider the multiplicative semigroup
$\Ar(C)$ generated by $\{L(s,\sigma_X)\;|\;X\in\mathcal X\}$. Obviously, 
$\Ar(C)$ is a subsemigroup of $\Ar$. Also, we consider 
$$\Hol(C,s_0) = \Hol(s_0)\cap \Ar(C), $$
the semigroup of the L-functions associated to $C$, which are holomorphic at $s_0$.

Assume that $\mathcal X=\{X_1,\ldots,X_m\}$. For $1\leq i\leq m$, we have that:
$$ F_i :=L(s,\sigma_{X_i})=\prod_{\chi_j \in X_i} f_j^{d_j},$$
where $d_j:=\chi_j(1)$, $1\leq j\leq r$.
The semigroup $\Ar(C)$ is generated by $F_1,\ldots, F_m$.
It follows that $F_1,\ldots,F_m$
are also multiplicatively independent, hence the semigroup $\Ar(C)$ is factorial of rank $m$,
i.e. it is isomorphic to $\mathbb Z_{\geq 0}^m$.

For $1\leq i\leq m$, let $\ell_i=\ord(F_i)$, where $\ord(F_i)$ denotes the order of the meromorphic function 
$F_i$ at $s_0$. We have that:
$$\Hol(C,s_0)=\{F_1^{a_1}\cdots F_m^{a_m}\;:\;a_1\ell_1+\cdots+a_m\ell_m \geq 0 \;\}.$$
Hence, by Gordan's lemma, see for instance \cite[Lemma 2.9]{bruns}, the semigroup $\Hol(C,s_0)$ is finitely generated. 
See also the proof of \cite[Theorem 1]{numb}.

The supercharacter-theoretic variant of Artin's conjecture (or $C$-Artin conjecture) at $s_0$, see \cite[Conjecture 1]{wong}, can be stated as:
$\Hol(C,s_0) = \Ar(C)$.

\begin{prop}\label{doi}
Let $G$ be a finite group which is $C$-quasi-monomial. Then $G$ satisfies the $C$-Artin conjecture.
\end{prop}

\begin{proof}
Since $G$ is $C$-quasi-monomial, for any $k\in \{1,\ldots,m\}$, there exists some subgroups $H_1,\ldots,H_t \leqslant G$ and 
linear characters $\lambda_1,\ldots,\lambda_t$ of $H_1,\ldots, H_t$ such that:
$$\lambda_1^G+\cdots+\lambda_t^G=d\sigma_{X_k},$$ 
where $d$ is a positive integer. It follows that 
$$F_k^d = \prod_{i=1}^t L(\lambda_i^G,s)$$ 
is holomorphic at $s_0$, hence $F_k$ is holomorphic at $s_0$.
\end{proof}

\begin{obs}
\emph{If $G = \Gal(K/ \mathbb Q)$ is equipped with the maximal theory $C$, then, according to Proposition \ref{unu}(2), $G$ is 
$C$-quasi-monomial. Hence, from Proposition \ref{doi}, it follows that $G$ satisfies the $C$-Artin conjecture at $s_0$.
Note that the Artin L-functions attached to supercharacters with respect to the maximal theory are: 
$$L(s, 1_G) = \zeta(s)\text{ and }L(s, \Reg(G)-1_G) = \zeta_K(s)/\zeta(s).$$ 
By a result of Aramata and Brauer \cite{brauer}, we know that $\zeta_K(s)/\zeta(s)$ is holomorphic at $s_0$ and,
of course, the Riemann-zeta function $\zeta(s)$ is holomorphic on $\mathbb C\setminus\{1\}$.}
\end{obs}

\pagebreak

The following result generalizes Theorem \ref{t1}:

\begin{teor}\label{t11}
Let $G = \Gal(K/ \mathbb Q)$ and let $C=(\mathcal X,\mathcal K)$ be a supercharacter theory of $G$. If $G$ is 
$C$-almost monomial, then the following are equivalent:
\begin{enumerate}
\item[(1)] Supercharacter theoretic Artin conjecture is true: 
           $\Hol(C,s_0) = \Ar(C)$.
\item[(2)] The semigroup $\Hol(C,s_0)$ is factorial.
\end{enumerate}
\end{teor}

\begin{proof}
$(1)\Rightarrow (2)$ Since the semigroup $\Ar(C)$ is factorial, there is nothing to prove.

$(2)\Rightarrow (1)$ Suppose that supercharacter theoretic Artin conjecture is not true. Then, there exists $1\leq k\leq m$
such that \begin{equation}\label{e1} \ord(F_k)<0. \end{equation}
The Dedekind zeta function $\zeta_K$ of $K$ can be decomposed as 
\begin{equation}\label{e2} \zeta_K = \prod_{i=1}^r f_i^{d_i} =  F_1\cdots F_m,\end{equation}
Since $\zeta_K$ is holomorphic in $\mathbb C\setminus\{1\}$, it follows that 
\begin{equation}\label{e3} \ord(\zeta_k)\geq 0.\end{equation}
From \eqref{e1},\eqref{e2} and \eqref{e3} it follows that there exists $\ell\in\{1,\ldots,m\}$ such that $\ord(F_{\ell})>0$.
For $i\in\{1,\ldots,m\}$, let $$n_i:=\min\{m\;:\;\ord(F_{\ell}^mF_i)\geq 0\}.$$
Since $f_1,\ldots,f_r$ are multiplicatively independent, the functions $F_{\ell}^{n_1}F_1,\ldots, F_{\ell}^{n_m}F_m$ are
irreducible in $\Hol(C,s_0)$. The Hilbert
basis $\mathcal H$ of $\Hol(C,s_0)$ is the uniquely determined minimal system of generators
of $\Hol(C,s_0)$, hence $\Hol(C,s_0)$ is factorial if and only if 
$\mathcal H$ has $m$ elements. It follows that $$\mathcal H=\{F_{\ell}^{n_1}F_1,\ldots, F_{\ell}^{n_m}F_m\}.$$
From \eqref{e1} it follows that $n_k>0$. Since $G$ is $C$-almost monomial, there exist
some subgroups $H_1,\ldots,H_t$ of $G$ and linear characters $\lambda_1,\ldots,\lambda_t$ of $H_1,\ldots,H_t$ such that 
\begin{equation}\label{e4}
\lambda_1^G + \cdots +\lambda_t^G = \alpha_1\sigma_{X_1}+\cdots+\alpha_m\sigma_{X_m},
\end{equation}
where $\alpha_i\geq 0$ are integers, $\alpha_k>0$ and $\alpha_{\ell}=0$. By Class Field Theory, for any $i\in\{1,\ldots,n\}$,
the Artin L-function $L(s,\lambda_i^G)$ is a Hecke L-function, so it is holomorphic at $s_0$. Hence, the function
\begin{equation}\label{e5}
F:=\prod_{i=1}^t L(s,\lambda_i^G) = L(s, \lambda_1^G + \cdots +\lambda_t^G),
\end{equation}
is holomorphic at $s_0$. From \eqref{e4} and \eqref{e5} it follows that
$$F = F_1^{\alpha_1}\cdots F_m^{\alpha_m} \in \Hol(C,s_0).$$
Since $\alpha_{k}>0$ and $\alpha_{\ell}=0$ this contradicts the fact that $F$ is a product of elements of $\mathcal H$.
\end{proof}

The following result generalizes Theorem \ref{t2}:

\begin{teor}\label{t22}
Let $G=\Gal(K/\mathbb Q)$ and let $C=(\mathcal X,\mathcal K)$ be a supercharacter theory of $G$ with $\mathcal X=\{X_1,\ldots,X_m\}$.
If $G$ is $C$-almost monomial and $s_0$ is not a common zero for any two distinct 
L-functions $L(s,\sigma_{X_{\ell}})$ and $L(s,\sigma_{X_k})$, where $k\neq \ell \in \{1,\ldots,m\}$, then 
all Artin L-functions from $\Ar(C)$ are holomorphic at $s_0$, i.e. the supercharacter theoretic 
Artin conjecture is true at $s_0$.
\end{teor}

\begin{proof}
We assume that $s_0$ is a pole of $F_j$, that is $\ord(F_j)<0$. Since the Dedekind zeta function $\zeta_K=F_1\cdots F_m$ is holomorphic at $s_0$, 
there is an index $k\neq j$ such that $F_k(s_0)=0$. Since $G$ is $C$-almost monomial, there exist 
some subgroups $H_1,\ldots,H_t\leqslant G$ and $\lambda_1,\ldots,\lambda_t$ some linear characters on $H_1,\ldots,H_t$ such that
$$ \lambda_1^G + \cdots +\lambda_t^G = \alpha_1 \sigma_{X_1}+ \cdots +\alpha_m\sigma_{X_m}, $$
with $\alpha_j>0$ and $\alpha_k=0$. The L-function
$$L(s,\lambda_1^G + \cdots +\lambda_t^G )=F_1^{\alpha_1}\cdots F_{k-1}^{\alpha_{k-1}}\cdot F_{k+1}^{\alpha_{k+1}}\cdots F_m^{\alpha_m},$$
is holomorphic at $s_0$. Since $\alpha_j>0$ and $\ord(F_j)<0$, it follows that there exists some index $\ell \notin \{j,k\}$ such that
$F_{\ell}(s_0)=0$, which contradicts the hypothesis.
\end{proof}

\textbf{Acknowledgment.} We gratefully acknowledge the use of the computer algebra system GAP (\cite{gap}) for our experiments.

{}

\vspace{2mm} \noindent {\footnotesize
\begin{minipage}[b]{15cm}
\emph{Mircea Cimpoea\c s}, Simion Stoilow Institute of Mathematics, Research unit 5, P.O.Box 1-764,
Bucharest 014700, Romania and University Politehnica of Bucharest, Faculty of
Applied Sciences, %Department of Mathematical Methods and Models, 
Bucharest, 060042, Romania.\\ 
 E-mail: mircea.cimpoeas@imar.ro
\end{minipage}}

\vspace{2mm} \noindent {\footnotesize
\begin{minipage}[b]{15cm}
\emph{Alexandru Florin Radu}, University Politehnica of Bucharest, Faculty of
Applied Sciences, %Department of Mathematical Methods and Models, 
Bucharest, 060042, Romania.\\ 
 E-mail: sasharadu@icloud.com
\end{minipage}}
\end{document}